\documentclass[12pt]{amsart}

\usepackage[utf8]{inputenc}

\usepackage[T1]{fontenc}

\usepackage[a4paper, margin=2.7cm]{geometry}

\usepackage{amsmath, amsthm, amsfonts, amssymb}
\usepackage{mathrsfs}           

\usepackage[colorlinks=true,linkcolor=blue,citecolor=blue,urlcolor=blue,breaklinks]{hyperref}

\usepackage{bbm}

 \newcommand{\R}{{\mathbb R}}
\newcommand{\T}{{\mathbb T}}                   
\newcommand{\C}{{\mathbb C}}                  

\newcommand{\A}{{\mathbb A}}

\def\al{\alpha}
\def\om{\omega}
\def\Om{\Omega}
\def\ga{\gamma}
\def\si{\sigma}

\def\la{\lambda}

\newtheorem{theorem}{Theorem}[section]
\newtheorem{corollary}[theorem]{Corollary}

\newtheorem{proposition}[theorem]{Proposition}

\theoremstyle{definition}
\newtheorem{definition}[theorem]{Definition}
\theoremstyle{remark}
\newtheorem{remark}[theorem]{Remark}
\theoremstyle{example}

\def\calL{{\mathcal{L}}}

\def\calP{{\mathcal{P}}}

\def\calR{{\mathcal{R}}}
\def\calL{{\mathcal{L}}}

\def\calP{{\mathcal{P}}}
\def\R{\mathbb R}

\def\C{\mathbb C}

\def\T{\mathbb T}

\title[Admissibility of observation operators]{On the admissibility of observation operators in the context of maximal regularity}

\author{O. El Mennaoui}
\address{Department of Mathematics, Faculty of Sciences Hay Dakhla, BP8106, 80000--Agadir, Morocco}
\email{o.elmennaoui@uiz.ac.ma}

\author{S. Hadd}
\address{Department of Mathematics, Faculty of Sciences Hay Dakhla, BP8106, 80000--Agadir, Morocco}
\email{s.hadd@uiz.ac.ma}

\author{Y. Kharou}
\address{Department of Mathematics, Faculty of Sciences Hay Dakhla, BP8106, 80000--Agadir, Morocco}
\email{yassine.kharou@edu.uiz.ac.ma}



\begin{document}

\maketitle

\begin{abstract}
  We study admissible observation operators for perturbed evolution equations using the concept of maximal regularity. We first show the invariance of the maximal $L^p$-regularity under non-autonomous Miyadera-Voigt perturbations. Second, we establish the invariance of admissibility of observation operators under such a class of perturbations. Finally, we illustrate our result with two examples, one on a non-autonomous parabolic system, and the other on an evolution equation with mixed boundary conditions and a non-local perturbation.
\end{abstract}


\section{Introduction}\label{sec:1}


In this work, we propose a study of admissibility of observation operators for non-autonomous linear systems within the framework of the maximal $L^p$-regularity. This problem is studied in \cite[Section 2]{S02} (see also \cite{H06}, \cite{HP94}, \cite{JDP95}) in an abstract way using properties of evolution families. We also note that the admissibility for non-autonomous systems is less understood compared with the autonomous case, see e.g., \cite{JP04}, \cite{L03}, \cite{Sala87},\cite{Staf05},\cite{TW09},\cite{W89}.

Before summarizing precisely our results, we first recall some definitions. Denote by $(X,\|\cdot\|)$, $(D,\|\cdot\|_D)$ and $(Y,\|\cdot\|_Y)$ three Banach spaces such that $D$ is continuously and densely embedded into $X$ ($ D \hookrightarrow_{d} X $), and let $A:D\to X$ be the generator of a strongly continuous semigroup $\T:=(\T(t))_{t\ge 0}$ on $X$. An operator $C\in\calL(D,Y)$ is called an $L^\theta$-admissible observation operator for $A$ (or $(C,A)$ is $L^\theta$-admissible) and we write $C\in \mathscr{O}^\theta_Y(A)$ if for a real number $\theta > 1$, and for some (hence all) $\al>0,$ there exists a constant $\ga>0$ such that
\begin{align}\label{CT-estim}
\left(\int^\al_0 \|C\T(t)x\|^\theta_Ydt\right)^{\frac{1}{\theta}}\le \ga \|x\|,\qquad (x\in D).
\end{align}
This kind of operators appears naturally if one observes the semigroup solution at the boundary of a domain $\Om\subset \R^n$ (when $X$ is a space of functions defined on $\Om$), see \cite[Chapter 4]{TW09}. We note that the admissibility in the autonomous case is a problem of integrability at $0$ of the function $t\mapsto C\T(t)x$.

In \cite[Definition 2.4]{S02}, Schnaubelt extended the $L^\theta$-admissibility from semigroups to abstract evolution families $(U(t,s))_{t\ge s\ge 0}$ on $X$, where extra assumptions are added to justify the expression $CU(t,s)x$. Recently, the work \cite{Kharou} studied the well-posedness of the following observed evolution system
\begin{align}\label{nCP}
\begin{cases}
\dot{u}(t)+A(t)u(t)=0,\quad u(s)=x,& t\in [s,\tau],\cr y(t)=Cu(t),& t\in [s,\tau],
\end{cases}
\end{align}
for $\tau > 0$ fixed, where $(A(t))_{t\in [0,\tau]}$ is a family of unbounded operators with constant domain $D$ such that $A(\cdot):[0,\tau]\to\calL(D,X)$ is a bounded strongly measurable application, $C:D\to Y$ is a linear unbounded observation operator and $s \in [0,\tau)$. The main assumptions used in \cite{Kharou} was for each $t\in [0,\tau]$, $A(t)$ has the property of maximal regularity (we write $A(t)\in \mathscr{MR}$, see Definition \ref{MR-singel}) and $A(\cdot):[0,\tau]\to\calL(D,X)$ satisfies the relative $\nu$-Dini assumption for some $\nu\in (1,\infty)$ (see the condition {\bf(H1)} in Section \ref{sec:2}). According to Theorem \ref{Kharou-thm1}, the evolution equation \eqref{nCP} is solved by a unique bounded evolution family $U:=(U(t,s))_{0\le s\le t\le \tau}$ on $X$ such that $U(t,s)x \in D$ for all $x \in X$ and for almost every $t \in [s,\tau)$.  In this case, the $L^\theta$-admissibility of the non-autonomous system \eqref{nCP} (or $(C,A(\cdot))$ is exactlty defined like \eqref{CT-estim}, see Definition \ref{Admissibility-non-auto}.



In the setting of the maximal regularity, the third named author (see Theorem \ref{Kharou-thm2} below) proved that under the condition {\bf(H1)}, $(C,A(\cdot))$ is admissible if and only if $(C,A(t))$ is admissible for each $t\in [0,\tau)$ (in the sense of \eqref{CT-estim}, as each $A(t)$ generates an analytic semigroup on $X$). This means that instead of checking the admissibility condition for the evolution family $ U $, it suffices to check this condition individually for each $ A (t) $ with respect to their associated semigroups.


In the autonomous case, it was proved in \cite{HI06} that $C$ is admissible for $A$ if and only if $C$ is admissible for $A+P$ for any admissible perturbation $P\in \mathscr{O}^\theta_X(A)$, i.e., $\mathscr{O}^\theta_Y(A)=\mathscr{O}^\theta_Y(A+P)$ for any $P \in \mathscr{O}^\theta_X(A)$. The main objective of the present work is to extend this result to non-autonomous systems. To this end, consider a  strongly measurable and bounded function $P(\cdot):[0,\tau]\to\calL(D,X)$ satisfying the condition {\bf (H2)} (see Section \ref{sec:3}). In addition, we denote by $A^P(\cdot)$ the function $A^P(\cdot):[0,\tau]\to \calL(D,X)$ given by
\begin{align}\label{perturbed-operator}
A^{P}(t):=A(t)+P(t),\quad t\in [0,\tau].
\end{align}
In Theorem \ref{maximal regularity of A^P} (Section \ref{sec:3}), we show that the conditions {\bf (H1)} and {\bf (H2)} imply  $A^{P}(\cdot) \in \mathscr{MR}_p(0,\tau)$. We mention that the same result is obtained in \cite{ABDH20}, using a quite different proof mainly based on Lebesgue extensions of the operators $P(t)$. Now the fact that $A^{P}(\cdot) \in \mathscr{MR}_p(0,\tau)$ shows that $A^P(\cdot)$ is associated with an evolution family $V:=(V(t,s))_{0\le s\le t\le \tau}$ on the trace space, due to \cite[Proposition 2.3]{ACFP07}. In Proposition \ref{Big-U} we prove that the evolution family $V$ has a unique extension to an evolution family on the whole space $X$.


Under the above conditions, we prove in Theorem \ref{main1} that $(C,A(\cdot))$ is admissible if and only if $(C,A^P(\cdot))$ is admissible. This is somehow a non-autonomous version of a result in \cite{HI06}.

The rest of the paper is organized as follows. Section \ref{sec:2} gathers notation and required
material from the concept of maximal $L^p$-regularity. Section \ref{sec:21} discuss the invariance of  maximal $L^p$-regularity under non-autonomous Miyadera-Voigt perturbations. In Section \ref{sec:3}, we prove results on admissibility of observation operators for perturbed evolution equations. Section \ref{sec:4} is devoted to examples which illustrate the main results.

\section{Preliminaries on the  maximal $L^p$-regularity}\label{sec:2}
Throughout this section, $X$ and $D$ are Banach spaces with $ D \hookrightarrow_{d} X $. Moreover, we take real numbers $p\in (1,\infty)$ and $\tau,\tau'>0$ such that $\tau'<\tau$.

We recall some background about the concept of  maximal $L^p$-regularity for  non-autonomous linear systems.
\begin{definition}\label{evolutions-families}
A family $U:=(U(t,s))_{0\le s\le t\le \tau}\subset\calL(X)$ is an evolution family on $X$ if
\begin{itemize}
  \item [{\rm (i)}] $U(t,s)=U(t,r)U(r,s),\quad U(s,s)=I$,\quad for any $s,t,r\in [0,\tau]$ with $0\le s\le r\le t\le \tau$,
  \item [{\rm (ii)}] $\left\{(t,s) \in [0,\tau]^2 : t \geq s \right\} \ni (t,s)\mapsto U(t,s)$ is strongly continuous.
\end{itemize}

\end{definition}
\noindent The concept of evolution families is used to solve some classes of evolution
equations (see, e.g., \cite[Chapter 2]{CL99}, \cite[Section VI.9]{EN00} and \cite{P83}).
That is, the evolution family arises as the solution operator of the well-posed non-autonomous
evolution equation \begin{align*}\dot{u}(t)+ A(t)u(t) = 0,\; u(s) = x\in X,\quad 0\le s\le t\le \tau,\end{align*} where
$A(t)$ is (in general) an unbounded linear operator for every fixed $t$. We mention
that, in general, the function $U(t,s)x$, as a function of $t$, is not differentiable. However,
the differentiability of such a function is guaranteed if $A(t),\;t\in [0,\tau],$ satisfy additional conditions (see e.g. \cite{AT87} and \cite{S01}).

In the following, we will introduce a situation in which one can naturally associate an evolution family to a family of unbounded operators. We first define the notion of maximal regularity for single operators.
\begin{definition}\label{MR-singel}
   We say that an operator $A \in \mathcal{L}(D,X)$ has $L^p$-maximal regularity ($p \in (1,\infty)$) and we write $A \in \mathscr{MR}_{p}$ if for all bounded interval $(a,b)\subset \R$ ($a<b$) and all $f \in L^{p}(a,b;X)$, there exists a unique $u \in W^{1,p}(a,b;X) \cap L^{p}(a,b;D)$ such that
\begin{equation}
\dot{u}(t) + A u(t) = f(t) \quad t\text{-a.e. on } [a,b], \qquad u(a) = 0.
\end{equation}
\end{definition}
According to \cite{D00},  the property of maximal $L^p$-regularity is independent of the bounded interval $(a,b)$, and if $A \in \mathscr{MR}_{p}$ for some $p \in (1,\infty)$ then $A \in \mathscr{MR}_{q}$ for all $q \in (1,\infty)$ . Hence, we can write $A \in \mathscr{MR}$ for short. Also it is well known that if $A\in \mathscr{MR},$ then $-A$ generates an analytic semigroup on $X$ \cite{KW04}. The converse is true if we work in Hilbert spaces \cite[Corollary 1.7]{KW04}. We also refer to \cite{DHP03,KW04} for the concept of $\calR$-Boundedness and Fourier Multipliers applied to maximal regularity.

Let $p \in (1,\infty)$ and define the following functional space
\begin{align*}
{\rm MR}_{p}(a,b) := W^{1,p}(a,b;X) \cap L^{p}(a,b;D),\end{align*}
which we call the space of maximal regularity. It is equipped with the following norm
\begin{align*} \| u \|_{\mathrm{MR}_{p}(a,b)} :=  \| u \|_{W^{1,p}(a,b;X)} + \| u \|_{L^{p}(a,b;D)} \qquad \left( u \in {\rm MR}_{p}(a,b) \right),\end{align*} or, with the equivalent norm $$ \left( \| u \|_{W^{1,p}(a,b;X)}^p + \| u \|_{L^{p}(a,b;D)}^p \right)^{1/p}.$$ The space ${\rm MR}_{p}(a,b)$ is a Banach space when equipped with one of the above norms. Moreover, we consider the \emph{trace space} defined by $$ {\rm Tr_{p}} := \left \{ u(a) : u \in {\rm MR}_{p}(a,b) \right \},$$ and endowed with the norm $$\| x \|_{\rm Tr_p} := \inf \left\{ \| u \|_{{\rm MR}_{p}(a,b)} : x = u(a) \right\}.$$ The space ${\rm Tr_p}$ is isomorphic to the real interpolation space $\left(X,D\right)_{1-1/p,p}$ \cite[Chapter 1]{L95}. In particular, ${\rm Tr_p}$ does not depend on the choice of the interval $(a,b)$ and $ D \hookrightarrow_{d} {\rm Tr_p} \hookrightarrow_{d} X $. We also note that $$ {\rm MR}_{p}(a,b) \hookrightarrow_{d} C \left( [a,b];{\rm Tr_p} \right) ,$$ and the constant of the embedding does not depend on the interval $(a,b)$ \cite[chap. 3]{A95}. In fact, for each $u\in {\rm MR}_{p}(a,b)$ and $t\in [a,b],$ we select $w(r)=u(t-r)$ for a.e. $r\in [0,t]$. As $w(0)=u(t)$ and $w\in {\rm MR}_{p}(0,t),$ then $u(t)\in {\rm Tr_p}$. On the other hand, by the definition of the trace space, we have
\begin{align*}
\|u(t)\|_{\rm Tr_p}\le \|w\|_{{\rm MR}_{p}(0,t)}\le \|u\|_{{\rm MR}_{p}(a,b)}.
\end{align*}

Consider now the  non-autonomous evolution equation
\begin{equation}\label{homogeneous initial value problem}
\dot{u}(t) + A(t) u(t) = 0 \quad t \text{ a.e. on } [s,\tau], \qquad u(s) = x
\end{equation}
for any $s \in [0,\tau)$, $x \in X$ and $A(\cdot) : [0,\tau] \to \mathcal{L}(D,X)$ is a strongly measurable and bounded operator-valued function such that $A(t) \in \mathscr{MR}$ for all $t \in [0,\tau]$.

\begin{definition}\label{Def-nice}
Let $p \in (1,\infty)$. We say that $A(\cdot)$ has $L^p$-maximal regularity on the bounded interval $[0,\tau]$ (and we write $A(\cdot) \in \mathscr{MR}_{p}(0,\tau)$), if and only if for all $[a,b]$ a sub-interval of $[0,\tau]$ and all $f \in L^{p}(a,b;X)$, there exists a unique $u \in \mathrm{MR}_{p}(a,b)$ such that
\begin{equation}
\dot{u}(t) + A(t) u(t) = f(t) \quad t\text{ a.e. on } [a,b], \qquad u(a) = 0.
\end{equation}

\end{definition}
\noindent Note that, by compactness of $[0,\tau]$, $A(\cdot) \in \mathscr{MR}_{p}(0,\tau)$ if and only if there exists $\alpha_0 > 0$ such that for all $[a,b]$ a sub-interval of $[0,\tau]$ with $|b-a| < \alpha_0$, and all $f \in L^{p}(a,b;X)$, there exists a unique $u \in \mathrm{MR}_{p}(a,b)$ such that $$ \dot{u}(t) + A(t) u(t) = f(t) \quad t\text{ a.e. on } [a,b], \qquad u(a) = 0. $$

In principle, one can ask if the condition $A(t) \in \mathscr{MR}$ for each $t \in [0,\tau]$ will imply that the evolution equation \eqref{homogeneous initial value problem} is solved by an evolution family on $X$. This is true if $A(\cdot)$ satisfies a kind of continuity.
\begin{definition}\label{Relative-cont}
We say that  $A(\cdot):[0,\tau]\to \calL(D,X)$ is relatively continuous if for all $\varepsilon > 0$ there exist $\delta > 0$ and $\eta \geq 0$ such that for all $s,t \in [0,\tau]$, we have $$ \| A(t) x - A(s) x \| \leq \varepsilon \| x \|_D + \eta  \| x \| \qquad (x \in D) $$ whenever $|t-s| \leq \delta$.
\end{definition}
Now, if $A(t)\in \mathscr{MR}$ for each $t\in [0,\tau]$ and $A(\cdot)$ is relatively continuous, the authors of \cite{ACFP07} proved the existence of an evolution family $U:=(U(t,s))_{0\le s\le t\le \tau}$ on the trace space $\mathrm{Tr}_p$. If the operators $A(t)$, $t \in [0,\tau]$ are, in addition, accretive, they further showed that $U$ extends to a contractive evolution family on $X$.

We need the following regularity on the map $A(\cdot)$.
\begin{definition}
Let $\nu \in (1,\infty)$. The function $A(\cdot):[0,\tau]\to\calL(D,X)$ satisfies the relative $\nu$-Dini condition if there exist $\eta \geq 0$, $\omega : [0,\tau] \to [0,\infty)$ a continuous function with $\omega(0)=0$ and
\begin{align}\label{omega-condition}
\int_{0}^{\tau} \left( \frac{\omega(t)}{t} \right)^{\nu} \mathrm{d}t < \infty \end{align} such that for all $x \in D$, $s,t \in [0,\tau]$, we have: $$\| A(t) x - A(s) x \| \leq \omega \left( | t-s | \right) \| x \|_D + \eta  \| x \|.$$
\end{definition}
In the rest of this paper, we need the following condition:
\begin{itemize}
  \item [{\bf(H1)}] For any $t\in [0,\tau],$ $A(t)\in\mathscr{MR},$ and $A(\cdot):[0,\tau]\to\calL(D,X)$ satisfies the relative $\nu$-Dini condition for some $\nu\in (1,\infty)$.
\end{itemize}
\begin{remark}\label{referee-remark}
\begin{itemize}
  \item [{\rm (i)}] It is worth noting that if $A(\cdot)$ satisfies the relative $p$-Dini condition, then $A(\cdot)$ is relatively continuous.
  \item [{\rm (ii)}]  If $A(\cdot)$ is H\"older continuous, that is,                                                                                        \begin{align*}\|A(t)x-A(s)x\|\le |t-s|^\al \|x\|_D,\qquad t,s\in [0,\tau],\end{align*}   for some $\al\in (0,1)$, then $A(\cdot)$ satisfies the relative $\nu$-Dini condition for all $\nu\in (1,\frac{1}{1-\alpha})$.
\item [{\rm (iii)}] The condition {\bf(H1)} implies that $A(\cdot)\in \mathscr{MR}_q(0,\tau)$ for any $q\in (1,\infty),$ see \cite{ACFP07}.
  \end{itemize}
\end{remark}

The following result is a slightly modification of \cite[Theorem 3.3]{Kharou}.

\begin{theorem}\label{Kharou-thm1}
Assume that $A(\cdot): [0,\tau] \longrightarrow \mathcal{L}(D,X)$ satisfies the condition {\rm\bf(H1)} and  let $U:=(U(t,s))_{0\le s\le t\le \tau}$  be the associated evolution family on the trace space $\mathrm{Tr}_p$ for $p \in (1,\infty)$. Then the following assertions hold:
\begin{itemize}
  \item [{\rm (i)}] $U$ extends to a bounded evolution family on $X$.
  \item [{\rm (ii)}] For every $x\in X,$  the function $u$ given by $u(t) := U(t,0)x$ is the unique solution of the problem
\begin{align*}
\dot{u}(t) + A(t) u(t) = 0 \quad t\text{ a.e. on } [0,\tau], \qquad u(0) = x.
\end{align*}
\item [{\rm (iii)}]  For every $q \in (1,\nu),$ and $x \in X$, we have $v:t\mapsto v(t)= tU(t,0)x \in {\rm MR}_{q}(0,\tau)$ and
  \begin{align}\label{goud-estomate} \| v \|_{{\rm MR}_{q}(0,\tau)} \leq M \| x \|\end{align}
  for a constant $M \geq 0$, depending on $q$ but independent of $x \in X$. Moreover, the function $u$ given by $u(t) := U(t,0)x$ belongs to the space $$ C\left([0,\tau];X\right) \cap L^{q}_{loc}\left((0,\tau];D\right) \cap W^{1,q}_{loc}\left((0,\tau];X\right).$$
\end{itemize}
\end{theorem}
\begin{remark}
The condition {\bf(H1)} implies the existence of an evolution family $U$ on ${\rm Tr}_p$ for any $p\in (1,\infty)$, see \cite[Theorem 2.7 and Proposition 2.3]{ACFP07}.
\end{remark}

\section{The stability of  maximal $L^p$-regularity under non-autonomous Miyadera-Voigt kind of perturbations}\label{sec:21}

In this section we assume that the function $A(\cdot):[0,\tau]\to\calL(D,X)$ satisfies the condition {\bf(H1)}. Then  we discuss the maximal regularity a $A(\cdot)+P(\cdot),$ where $P(\cdot):[0,\tau]\to\calL(D,X)$ is a strongly measurable and bounded application such that:
\begin{itemize}
  \item [] {\bf(H2)} There exists $\mu\in (1,\infty)$ and a constant $c>0$ such that for $0\le s<\tau'<\tau$, we have
\begin{equation*}
\left(\int_{s}^{\tau'} \| P(t) U(t,s) x \|^{\mu} \, dt\right)^{\frac{1}{\mu}} \leq c \| x \|, \qquad ( x \in D).
\end{equation*}
\end{itemize}
The following result shows the maximal regularity of $A^{P}(\cdot):[0,\tau]\to \calL(D,X)$ defined in \eqref{perturbed-operator}.
\begin{theorem}\label{maximal regularity of A^P}
Assume that $A(\cdot)$ and $P(\cdot)$ satisfy the conditions {\bf(H1)-(H2)}. Then $A^{P}(\cdot) \in \mathscr{MR}_{q}(0,\tau)$ for every $q \in (1,\mu]$.
\end{theorem}
\begin{proof}
According to Remark \ref{referee-remark} we have $A(\cdot)\in\mathscr{MR}_q(0,\tau)$ for any $q\in (1,\infty)$. Let $[a,b]$ ($a<b$) be a sub-interval of $[0,\tau]$ and define the operators $\mathcal{A}$ and  $\mathcal{B}$ on $L^{q}(a,b;X)$ by:
\begin{align*} & D ( \mathcal{B} ) := \left\{ u \in W^{1,q}(a,b;X) : u(a) = 0 \right\},  \quad \mathcal{B}u := \dot{u}, \cr  & D ( \mathcal{A} ) := L^{q}(a,b;D),  \quad \mathcal{A}u := A(s) u(s) \; \left(s \in (a,b)\right).
\end{align*}
On the other hand, define
\begin{align*}
(\calP u)(t):=P(t)u(t),\quad u\in L^q([a,b],D),\quad a.e. t\in [a,b].
\end{align*}
Observe that for any $u\in L^q([a,b],D)$, $t\mapsto (\calP u)(t)$ is measurable and
\begin{align*}
\int^b_a \left\|(\mathcal{P}u)(t)\right\|^qdt\le \left(\sup_{t\in [0,\tau]}\|P(t)\|_{\calL(D,X)}\right)^q \|u\|^q_{L^q([a,b],D)}.
\end{align*}
This shows that $\calP:L^q([a,b],D)\to L^q([a,b],X)$, and \begin{align*}\|\calP u\|_{L^q([a,b],X)}\le \|P(\cdot)\|_\infty \|u\|_{L^q([a,b],D)},\qquad (u\in L^q([a,b],D)).
\end{align*}
Now let the operator $\mathcal{A}+\mathcal{B}+\mathcal{P}$ on $L^{q}(a,b;X)$ with domain $D( \mathcal{A}+\mathcal{B}+\mathcal{P} ) := D ( \mathcal{A} ) \cap D ( \mathcal{B} ) $. The fact that $A(\cdot)\in\mathscr{MR}_{q}(0,\tau)$ implies that the inverse $(\mathcal{A}+\mathcal{B})^{-1}$ exists in $L^q([a,b],X)$. Moreover, we can write
\begin{align*} \mathcal{A}+\mathcal{B}+\mathcal{P} = (I+\mathcal{P}(\mathcal{A}+\mathcal{B})^{-1})(\mathcal{A}+\mathcal{B}). \end{align*} Let $q\in (1,\mu]$ and $q'\in (1,\infty)$ such that $\frac{1}{q}+\frac{1}{q'}=1$. Then $P(\cdot)$ satisfies also the condition {\bf (H2)} if we replace $\mu$ with $q$. Now for $f \in L^{q}(a,b;X)$,
\begin{align*}
\| \mathcal{P}(\mathcal{A}+\mathcal{B})^{-1} f \|^q_{L^{q}(a,b;X)} &= \int_{a}^{b} \| [\mathcal{P}(\mathcal{A}+\mathcal{B})^{-1} f] (t) \|^q \, dt \\
 &= \int_{a}^{b} \left\| P(t) \int_{a}^{t} U(t,r) f(r) \, dr \right\|^q \, dt \\
 &\leq \int_{a}^{b} \left(  \int_{a}^{t}  \left\| P(t) U(t,r) f(r) \right\| \, dr \right)^q \, dt \\
&\leq (b-a)^{q/q'} \int_{a}^{b} \int_{a}^{t}  \left\| P(t) U(t,r) f(r) \right\|^q \, dr  \, dt \\
 &\leq (b-a)^{q/q'} \int_{a}^{b} \int_{r}^{b}  \left\| P(t) U(t,r) f(r) \right\|^q \, dt  \, dr \\
 &\leq (b-a)^{q/q'} c \int_{a}^{b} \left\| f(r) \right\|^q  \, dr\cr & = (b-a)^{q/q'} c\; \|f\|^q_{L^{q}(a,b;X)},
\end{align*}
by H\"older inequality, Fubini theorem and {\bf{(H2)}}. By choosing a small interval $[a,b],$ we can assume that $(b-a)^{q/q'} c<1$. This ends the proof due to the compactness of $[0,\tau]$ (see comments after Definition \ref{Def-nice}).
\end{proof}
\begin{remark}
A similar result for the maximal regularity of $A^P(\cdot)$ is obtained in \cite{ABDH20}, where Lebesgue extensions of the operators $P(t)$ are used to prove the result.
\end{remark}
The following proposition shows that to $A^{P}(\cdot)$, we can associate an evolution family $V:=(V(t,s))_{0\le s\le t\le \tau}$ on $X$ (compare with \cite[Proposition 4.3]{H06}).
\begin{proposition}\label{Big-U}
Assume that $A(\cdot)$ and $P(\cdot)$ satisfy the conditions {\bf(H1)-(H2)}. Then $A^{P}(\cdot)$ generates an evolution family $V:=(V(t,s))_{0\le s\le t\le \tau}$ on $X$ satisfying the following formula
\begin{align}\label{VCF-U}
V(t,s)x=U(t,s)x+\int^t_s V(t,\si)P(\si)U(\si,s)xd\si
\end{align}
for any $0\le s\le t\le \tau$ and $x\in D$. Furthermore for any $x\in X$,  the function $u:t \mapsto V(t,0)x$ solves the problem
belongs to the space $$ C\left([0,\tau];X\right) \cap L^{q}_{loc}\left((0,\tau];D\right) \cap W^{1,q}_{loc}\left((0,\tau];X\right),$$ for all $q \in (1,\inf\{\nu,\mu\}]$.
\end{proposition}
\begin{proof}
According to Theorem \ref{maximal regularity of A^P}, $A^{P}(\cdot) \in \mathscr{MR}_{q}(0,\tau)$ for every $q \in (1,\mu]$. Then by \cite[Lemma 2.2]{ACFP07}, to $A^{P}(\cdot)$, we associate a unique evolution family $V:=(V(t,s))_{0\le s\le t\le \tau}$ on ${\rm Tr_q}$ for $q \in (1,\mu]$. In particular, for $x\in D$, the functions $t \mapsto V(t,s)x$ and $t \mapsto U(t,s)x$ belong to $\mathrm{MR}_{q}(s,\tau)$ for $q \in (1,p_0]$, so that
\begin{align*}
\frac{d}{dt}\left( V(t,s)x -U(t,s)x\right)&=-A^{P}(t)V(t,s)x+A(t)U(t,s)x\cr &= -A^{P}(t)(V(t,s)x-U(t,s)x) - P(t)U(t,s)x.
\end{align*}
 The fact that $A^{P}(\cdot)\in\mathscr{MR}_q(0,\tau)$  for $q \in (1,\mu]$ implies that
there exists a constant $\kappa>0$ such that
\begin{align}\label{MR-estimate-BigU}
\|V(\cdot,s)x -U(\cdot,s)x\|_{\mathrm{MR}_q(s,\tau)}&\le \kappa \|P(\cdot)U(\cdot,s)x\|_{L^q([s,\tau],X)}\le \kappa \ga \|x\|
\end{align}
for any $x\in X$, due to {\bf(H2)}. By using the estimate \eqref{MR-estimate-BigU} and the embedding $\mathrm{MR}_q(s,\tau) \hookrightarrow_d C([s,\tau],X)$ (densely and continuously), there exists a constant $M>0,$ independent of $t$ and $s$ such that
\begin{align*}
\|V(t,s)x -U(t,s)x\|\le M \|x\|
\end{align*}
for any $0\le s\le t\le \tau$ and $x\in {\rm Tr}_p$. Now the density of ${\rm Tr}_p$ in $X$ implies that $V$ has an extension to an evolution family on $X$. Now we can write
\begin{align*}
V(t,s)x -U(t,s)x=\int^t_s V(t,\si)P(\si)U(\si,s)xd\si
\end{align*}
for any $0\le s\le t\le \tau$ and $x\in X$. \\ On the other hand, let $q\in (1,\inf\{\nu,\mu\}]$ and $x\in {\rm Tr}_q$. By the $L^q$-maximal regularity of $A(\cdot)$, the function $v(t):=tV(t,0)x$ is the unique solution of the non-homogeneous problem
\begin{align*}
\begin{cases}
\dot{v}(t)+A^P(t)v(t)=V(t,0)x,& a.e.\; t\in [0,\tau],\cr u(0)=0,
\end{cases}
\end{align*}
and  there exists a constant $\kappa>0$ such that
\begin{align*}
\|v\|_{{\rm MR}_q(0,\tau)}\le \kappa \|V(\cdot,0)x\|_{L^q(0,\tau;X)}\le M'\|x\|
\end{align*}
for a constant $M'>0$ depending on $\kappa,M$ and $q$. By the density of ${\rm Tr}_q$ in $X,$ the above estimate holds also for every $x\in X$. In addition, since ${\rm MR}_q(0,\tau)\hookrightarrow_d C([0,\tau],X)$, we have
\begin{align*}
\|t V(t,0)x\|\le M' \|x\|
\end{align*}
for almost every $t\in [0,\tau]$ and $x\in X$. In particular, for every $x\in X$ and every $q\in (1,\inf\{\nu,\mu\}],$ we have
\begin{align*}
V(\cdot,0)x\in L^q_{loc}((0,\tau],D)\cap W^{1,q}_{loc}((0,\tau],X).
\end{align*}
\end{proof}


\section{The stability of admissibility of observation operators under non-autonomous Miyadera-Voigt kind of perturbations}\label{sec:3}
Throughout this section $X,$ $Y$ and $D$ are Banach spaces as in Section \ref{sec:1}. We take real numbers $p\in (1,\infty)$ and $\tau,\tau'>0$ such that $\tau'<\tau$. Let $A(\cdot):[0,\tau]\to\calL(D,X)$ be a strongly measurable and bounded function, and $C:D\to Y$ is a linear (observation) operator. From Section \ref{sec:1}, if $A(\cdot)$ satisfies the condition {\bf(H1)}, then we can associate to it a bounded evolution family $U=(U(t,s))_{0\le s\le t\le \tau}$ on the whole space $X$.

\begin{definition}\label{Admissibility-non-auto}
Let the assumption {\bf(H1)} be satisfied. The operator $C$ is $L^\theta$-admissible for $A(\cdot)$ or ($(C,A(\cdot))$ is $L^\theta$-admissible) with $\theta\in (1,\infty)$ if and only if
\begin{equation}\label{admissibleobservation0}
\left(\int_{s}^{s+\al} \| C U(t,s) x \|_{Y}^{\theta} \, dt\right)^{\frac{1}{\theta}} \leq \gamma \| x \|
\end{equation}
for $s\in [0,\tau),$ $x\in D$, and some constant $\ga,\al>0$ such that $s\le s+\al<\tau$.
\end{definition}
The following result shows another reformulation of the admissibility estimate \eqref{admissibleobservation0}.
\begin{proposition}\label{equiv-admiss}
Let the assumption {\bf(H1)} be satisfied.  Then  $(C,A(\cdot))$ is $L^\theta$-admissible if and only if there exists $\ga>0$ such that for any $s\in [0,\tau),$ we have
\begin{equation}\label{admissibleobservation}
\left(\int_{s}^{\tau} \| C U(t,s) x \|_{Y}^{\theta} \, dt\right)^{\frac{1}{\theta}} \leq \gamma \| x \|, \qquad (x \in D).
\end{equation}
\end{proposition}
\begin{proof}
Obviously the condition \eqref{admissibleobservation} implies \eqref{admissibleobservation0}. Conversely, let $x\in D,$ $\theta\in (1,\infty),$ $t,s\in [0,\tau)$ and $t_0\in [s,t]$. Then
\begin{align*}
\int^t_s \|CU(r,s)x\|^\theta dr& =\int^{t_0}_s \|CU(r,s)x\|^\theta dr+ \int^t_{t_0} \|CU(r,s)x\|^\theta dr \cr & \le \ga^\theta \|x\|^\theta+  \int^t_{t_0} \|CU(r,s)x\|^\theta dr,
\end{align*}
due to \eqref{admissibleobservation0}. On the other hand, if we denote by $c$ the constant of the maximal regularity of $A(\cdot)$, then
\begin{align*}
 \int^t_{t_0} \|CU(r,s)x\|^\theta dr& \le \|C\|^\theta_{\calL(D,Y)} \int^t_{t_0} \|U(r,t_0)U(t_0,s)x\|^\theta_D dr\cr & \le \|C\|^\theta_{\calL(D,Y)} c^\theta \|U(t_0,s)x\|^\theta_{{\rm Tr}_\theta}\cr & \le \frac{(c\|C\|_{\calL(D,Y)})^\theta}{(t_0-s)^{\frac{\theta}{\theta'}}} \|x\|^\theta
\end{align*}
for $\theta'>1$ with $\frac{1}{\theta}+\frac{1}{\theta'}=1$. From the above two estimates, we deduce that
\begin{align*}
\int^t_s \|CU(r,s)x\|^\theta dr \le \left(\ga^\theta+\frac{(c\|C\|_{\calL(D,Y)})^\theta}{(t_0-s)^{\frac{\theta}{\theta'}}}\right) \|x\|^\theta.
\end{align*}
The result follows by choosing $t_0=s+\al$.
\end{proof}
The estimate \eqref{admissibleobservation0} seems to be a natural generalization of the estimate \eqref{CT-estim}, see also \cite[Definition 2.1]{HHO19}. On the other hand, one can see that $(C,A(\cdot))$ is $L^\theta$-admissible if and only if $(C,\la+A(\cdot))$ is for $\la\in\C$.

The fact that $(C,A(\cdot))$ is admissible is not easy to verify in the examples. In some cases, individual admissibility may suffice. The relation between the admissibility of observation operators for $A(\cdot)$ and for each single operator $A(t_0)$ ($t_0\in [0,\tau]$) is given in the following result which is taken from \cite[Theorem 3.8]{Kharou}.
\begin{theorem}\label{Kharou-thm2}
Let $A(\cdot): [0,\tau] \to \mathcal{L}(D,X)$ satisfy the condition {\bf(H1)}. Then $\left( C, A\left(\cdot\right) \right)$ is admissible, if and only if $\left(C,A(t)\right)$ is admissible for all $t \in [0,\tau)$.
\end{theorem}
The following proposition shows a condition on $A(\cdot)$ for which the property ``$\left(C,A(t)\right)$ is $L^\theta$-admissible for all $t \in [0,\tau]$''  is equivalent to ``there exists $t_0\in [0,\tau]$ such that $(C,A(t_0))$ is $L^\theta$-admissible''.
\begin{proposition}
Let $A(\cdot):[0,\tau]\to\calL(D,X)$ be a strongly measurable and bounded function. Assume that $A(t) \in \mathscr{MR}$ for every $t \in [0,\tau]$ and that there exist $M, \eta > 0$ and $\beta \in (1,\infty)$ such that
\begin{equation*}
\| A(t)x - A(s)x \| \leq M \| x \|_{{\rm Tr}_\beta} + \eta \|x\|, \qquad \left( t,s \in [0,\tau] ,  x \in D \right).
\end{equation*}
Then $$ \mathscr{O}^\theta_Y(A(t_0)) = \mathscr{O}^\theta_Y(A(t_1)) $$ for every $\theta < \frac{\beta}{\beta-1}$ and $t_0 , t_1 \in [0,\tau]$.
\end{proposition}

\begin{proof}
Let $0 < t \leq \tau$, $x \in D$ and $\theta \in (1,\infty)$. Consider the function $v:[0,t] \ni r \mapsto e^{-(t-r)A(t_0)} e^{-rA(t_1)} x$. Then $v \in W^{1,\theta}(0,t;X)$ and
\begin{align*}
\frac{d}{dr} v(r) &= A(t_0) e^{-(t-r)A(t_0)} e^{-rA(t_1)} x - e^{-(t-r)A(t_0)} A(t_1) e^{-rA(t_1)} x \\
 &= e^{-(t-r)A(t_0)} \left( A(t_0) - A(t_1) \right) e^{-rA(t_1)} x .
\end{align*}
By integrating between $0$ and $t$, we obtain $$ e^{-t A(t_1)} x = e^{-t A(t_0)} x + \int_{0}^{t} e^{-(t-r)A(t_0)} \left( A(t_0) - A(t_1) \right) e^{-rA(t_1)} x dr .$$
Thus, according to \cite[Prop.3.3]{H05}, the admissibility of $C$ for $A(t_0)$ implies
\begin{align}\label{key-estimate}
\begin{split}
\int_{0}^{\tau} \| C e^{-t A(t_1)} x &- C e^{-t A(t_0)} x \|^\theta dt  \cr & \leq \int_{0}^{\tau} \left\| C \int_{0}^{t}  e^{-(t-r)A(t_0)} \left( A(t_0) - A(t_1) \right) e^{-rA(t_1)} x dr \right\|^\theta dt \\
 &\leq k_{\tau} \int_{0}^{\tau} \| \left( A(t_0) - A(t_1) \right) e^{-rA(t_1)} x \|^\theta dr \\
 &\leq k_{\tau} \int_{0}^{\tau} \left( M \| e^{-rA(t_1)} x \|_{{\rm Tr}_\beta} + \eta \| e^{-rA(t_1)} x\| \right)^\theta dr \\
 &\leq 2^\theta M^\theta k_{\tau} \int_{0}^{\tau} \frac{1}{r^{\theta\frac{\beta-1}{\beta} }} dr \|x\|^\theta + 2^\theta \eta^\theta c_{\tau}^\theta \|x\|^\theta.
\end{split}
\end{align}
\end{proof}
\begin{remark}\label{B-admissible}
Let $A(\cdot):[0,\tau]\to\calL(D,X)$ be a strongly measurable and bounded function, $\theta\in (1,\infty)$ and $B\in \mathscr{O}^\theta_{Y}(A(t))$ for any $t\in [0,\tau]$. Assume that $A(t) \in \mathscr{MR}$ for every $t \in [0,\tau]$ and that there exist $M, \eta > 0$  such that
\begin{align}\label{B-estimate}
\| A(t)x - A(s)x \| \leq M \| Bx \|+ \eta \|x\|, \qquad \forall t,s \in [0,\tau] , \forall x \in D .
\end{align}
Then $\mathscr{O}^\theta_Y(A(t_0)) = \mathscr{O}^\theta_Y(A(t_1))$ for any $t_0,t_1\in [0,\tau]$. Indeed, assume that $C\in \mathscr{O}^\theta_Y(A(t_0))$. By combining \eqref{key-estimate} with \eqref{B-estimate}, we obtain
\begin{align*}
\int_{0}^{\tau} \| C e^{-t A(t_1)} x \|^\theta dt&\le  \gamma_{\tau}^\theta \|x\|^\theta + k_{\tau}\int_{0}^{\tau} \left( M\| Be^{-rA(t_1)} x \| + \eta \| e^{-rA(t_1)} x\| \right)^\theta dr\cr \le & c \|x\|^\theta,
\end{align*}
for a constant $c:=c(\theta,\tau)>0,$ due to the admissibility of $B$ for $A(t_1)$.
  \end{remark}
The following is the main result of this paper, which gives the invariance of admissibility of observation under unbounded perturbations.
\begin{theorem}\label{main1}
Assume that $A(\cdot)$ and $P(\cdot)$ satisfy the conditions {\bf(H1)-(H2)}. Then
\begin{align*}
\mathscr{O}^\theta_{Y}\left( A(\cdot) \right) = \mathscr{O}^\theta_{Y}\left( A^{P}(\cdot) \right) \qquad (1<\theta \leq \mu),
\end{align*}
where $\mu$ is from {\bf(H2)}.
\end{theorem}
\begin{proof}
Let $C\in \mathscr{O}^\theta_{Y}\left( A(\cdot) \right)$, $x\in D$ and $s\in [0,\tau')$ with $\tau'<\tau$. According to Proposition \ref{Big-U}, we have
\begin{align*}
CV(t,s)x=C(V(t,s)x-U(t,s)x)+CU(t,s)x
\end{align*}
for $0\le s\le t\le \tau$. Thus
\begin{align*}
\int_s^{\tau'} \|CV(t,s)x\|^\theta dt&\le 2^p \|C\|^\theta_{\calL(D,Y)} \int^{\tau'}_s \|V(t,s)x-U(t,s)x\|^\theta_D dt+ (2\beta)^\theta \|x\|^\theta\cr & \le 2^\theta \|C\|^\theta_{\calL(D,Y)}\|V(\cdot,s)x -U(\cdot,s)x\|^\theta_{\mathrm{MR}_\theta(s,\tau')}+ (2\beta)^\theta \|x\|^\theta\cr & \le (2\kappa \ga \|C\|_{\calL(D,Y)})^\theta \|x\|^p+(2\beta)^\theta \|x\|^\theta:= \delta^\theta \|x\|^\theta,
\end{align*}
due to \eqref{MR-estimate-BigU}. \\ Conversely, assume that $C\in \mathscr{O}^\theta_{Y}\left( A^{P}(\cdot) \right)$, $x\in D$ and $s\in [0,\tau)$. By the same computation if we replace $C$ with $P(t)$ and using the boundedness of the map  $t\in [0,\tau]\mapsto \|P(t)\|_{\calL(D,Y)}$, we can see that
\begin{align*}
\left(\int_s^{\tau'} \|P(t)V(t,s)x\|^{\mu}dt\right)^{\frac{1}{\mu}}\le \tilde{\ga} \|x\|.
\end{align*}
Hence $P(\cdot)$ satisfies {\bf(H2)} with respect to $A(\cdot)+P(\cdot)$.  We can then use the first case to show that $C\in \mathscr{O}^\theta(A(\cdot))$.
\end{proof}

\begin{remark}\label{R1}
Let the assumptions of Theorem \ref{main1} be satisfied. Let $C(\cdot):[0,\tau]\to\calL(D,Y)$ be a strongly measurable and bounded map and let $\theta\in (1,\mu]$. Using the same arguments as in the proof of Theorem \ref{main1}, we can prove the following result: for $s\in [0,\tau')$ and $\tau'<\tau$, there exists $\ga>0$ such that
\begin{align}\label{CCC}
  \int^{\tau'}_{s}\|C(t)U(t,s)x\|^\theta dt\le \ga^\theta \|x\|^\theta \qquad (x \in D)
  \end{align}
  if and only if for $s\in [0,\tau')$ and $\tau'<\tau$, there exists $c>0$ such that
\begin{align*}
  \int^{\tau'}_{s}\|C(t)V(t,s)x\|^\theta dt\le c^\theta \|x\|^\theta \qquad (x \in D).
  \end{align*}
\end{remark}

\begin{remark}\label{R2}
Let the assumptions of Theorem \ref{main1} be satisfied. For any $a\in [0,\tau],$ we denote by $\T^a:=(\T^a(t))_{t\ge 0}$ the analytic semigroup generated by $(-A(a),D)$. Let $C(\cdot):[0,\tau]\to\calL(D,Y)$ be a strongly measurable and bounded map. Does the condition \eqref{CCC} is equivalent to $(C(a),A(a))$ is admissible for any $a$? From the proof of \cite[Theorem 3.3]{Kharou}, we deduce that $U(\cdot,a)x-\T^a(\cdot-a)x\in MR_p(0,\tau)$ for any $p\in (1,\nu]$,  $x\in {\rm Tr}$, and
\begin{align}\label{PPPP}
\left\|U(\cdot,a)x-\T^a(\cdot-a)x\right\|_{MR_p(a,\tau)}\le M \|x\|
\end{align}
for a certain constant $M>0$. Now let $b>a$ and $x\in D\subset {\rm Tr}$. Then
\begin{align*}
& \int^{b-a}_0 \left\|C(a)\T^a(t)x\right\|^pdt= \int^{b}_a \left\|C(a)\T^a(t-a)x\right\|^pdt\cr & \qquad \le 3^p \left( \int^b_a \left\|(C(a)-C(t))\T^a(t-a)x\right\|^pdt \right. \cr  &\qquad\qquad \left. +\int^b_a \left\|C(t)(\T^a(t-a)x-U(t,a)x)\right\|^pdt +\int^b_a \left\|C(t)U(t,a)x\right\|^pdt \right).
\end{align*}
Using the estimate \eqref{PPPP} and \eqref{CCC}, we have
\begin{align}\label{E1}
\begin{split}
& \int^\al_0 \left\|C(t)(\T^a(t)x-U(t,0)x)\right\|^pdt \le (\kappa M)^p \|x\|^p,\cr
&  \int^\al_0 \left\|C(t)U(t,0)x\right\|^pdt\le \ga^p \|x\|^p,
\end{split}
\end{align}
where $\kappa:=\sup_{t\in [0,T]}\|C(t)\|_{\calL(D,X)}$. On the other hand, if we assume that $C(\cdot),$ is relatively  $\mu$-Dini, then for any $p\in (1,\nu],$
\begin{align}\label{E2}
& \int^b_a \left\|(C(a)-C(t)) \T^0(t-a)x\right\|^pdt \cr  &\qquad\qquad\le \int^b_a \om(b-a)^p \|\T^a(t)x\|_D dt+\eta^p \int^b_a \|\T^a(t)x\|^p dt\cr & \qquad\qquad
\le \int^b_a \left(\frac{\om(t-a)}{t-a}\right)^p \|(t-a)\T^a(t-a)x\|_D dt+ c \|x\|^p\cr & \qquad\qquad
\le \int^b_a \left(\frac{\om(t-a)}{t-a}\right)^\nu \|(t-a)\T^a(t-a)x\|_D dt+ c \|x\|^p\cr & \qquad\qquad \le \beta^p \|x\|^p
\end{align}
for a constant $\beta>0$. Now the admissibility of $(C(a),A(a))$ follows by combining \eqref{E1} and \eqref{E2}.
\end{remark}

\begin{corollary}\label{cor-de-main}
Let $A_0,A_1:D\to X$ be two generators on $X$ such that $A_0\in \mathscr{MR}$ and
\begin{align}\label{nice}
\|(A_1-A_0)x\|\le M \|x\|_{{\rm Tr}_2}+\eta \|x\|
\end{align}
for and $x\in D$ and for some constants $M,\eta>0$. Then $\mathscr{O}^\theta_Y(A_0)=\mathscr{O}^\theta_Y(A_1)$ for any $\theta\in (1,2)$.
\end{corollary}
\begin{proof}
We set $A(t)=A_0$ and $P(t)=A_1-A_0$ for any $t\in [0,\tau]$. Clearly $A(\cdot)$ satisfies the condition {\bf (H1)}. Let us  now verify that $P(\cdot)$ satisfies the condition {\bf (H2)} for $\mu=\theta\in (1,2)$. Let $x\in D$ and $s,\tau'\in [0,\tau)$ with $s<\tau'$. By using \eqref{nice}, a change of variables and trace space properties, we obtain
\begin{align*}
\int^{\tau'}_s \|P(t)e^{-(t-s)A_0} x\|^\theta dt &\le 2^\theta\left( M^\theta \int^{\tau'}_0 \frac{dt}{t^{\frac{\theta}{2}}} dt +\eta^\theta \right) \|x\|^\theta\cr & \le \tilde{\ga}^\theta \|x\|^\theta
\end{align*}
for a constant $\tilde{\ga}:=\tilde{\ga}(\theta,\tau',M,\eta)>0$. Now Theorem \ref{main1}, shows that $\mathscr{O}^\theta_Y(A_0)=\mathscr{O}^\theta_Y(A(\cdot))=\mathscr{O}^\theta_Y(A(\cdot)+P(\cdot))=\mathscr{O}^\theta_Y(A_1).$
\end{proof}

\section{Examples} \label{sec:4}
In this section, we give two situations that illustrate our abstract results. We choose the heat equation as model, but the results can also be applied to other physical models.
\subsection{A heat equation with Dirichlet conditions and  point observation}Let $p,q\in (1,\infty)$ and $n\in\mathbb{N}^\ast$ such that $\frac{1}{p}+\frac{1}{q}=1$ and $q>\frac{np}{2}$. Let $\Om$ be a bounded open subset of $\R^n$ with $C^2$-boundary $\partial\Om$, and let $a_{kl},a_k,a_0:\R^+\times \overline{\Om}\to \R,$ $k,l=1,\cdots,n$ be bounded continuous functions such that
\begin{align*}
\sum_{ij}a_{ij}(t,x)x_i x_j\ge \beta \|x\|^2
\end{align*}
for a constant $\beta>0$. In addition, we suppose that
\begin{align*}
\max_{1\le i,j\le n}\sup_{x\in \Om}|a_{ij}(t,x)-a_{ij}(s,x)|\le \om(t-s)
\end{align*}
for any $s,t\in [0,\tau],$ where $\om:[0,\tau]\to\R^+$ is a continuous function satisfying
\begin{align*}
\int^{\tau}_0 \left(\frac{w(t)}{t}\right)^\nu dt < \infty
\end{align*}
for some $\nu\in (1,\infty)$. Select
\begin{align*}X=L^q(\Om)\quad\text{and}\quad D=W^{2,q}(\Om)\cap W^{1,q}_0(\Om).\end{align*}
Define a family of unbounded operators on $X$, $A(t):D\to X$ with
\begin{align*}
A(t)u:=-\sum_{ij}a_{ij}(t,\cdot)\partial_i \partial_j u-b_0(t,\cdot)u,\quad t\in [0,\tau],\; u\in D,
\end{align*}
where $b_0\in L^\infty((0,\tau)\times \Om)$. According to \cite{ACFP07}, and \cite{DHP03} the operator $A(t)$ has the maximal $L^p$ regularity for any $t$. Remark that for any $t,s\in [0,\tau]$ and $u\in D,$
\begin{align*}
\|A(t)-A(s)\|_{L^q(\Om)}&\le \om(t-s)\sum_{ij}\|\partial_i \partial_ju\|_{L^q(\Om)}+2 \|b_0\|_\infty \|u\|_{L^q(\Om)}\cr & \le
\om(t-s)\|u\|_D+2 \|b_0\|_\infty \|u\|_{L^q(\Om)}.
\end{align*}
Thus $A(\cdot)$ satisfies the condition {\bf(H1)}. It is well-known that the evolution family $(U(t,s))_{0\le s\le t\le \tau}$ on $X$ associated with $A(t)$ is given by
\begin{align*}
U(t,s)\varphi=\int_{\Om} k(t,s,\cdot,x)\varphi(x)dx,\qquad 0\le s\le t \le  \tau,\quad \varphi\in X,
\end{align*}
for a continuous kernel $k(t,s,\cdot,x)$, $\tau\ge t>s\ge 0$ and $x\in \Om$ such that
\begin{align}\label{kernal}
|k(t,s,\xi,x)|\le M(t-s)^{-\frac{n}{2}} \exp\left( -\frac{\delta |\xi-x|^2}{t-s}+\tilde{\delta}(t-s) \right)
\end{align}
for $\xi\in \overline{\Om}$,  constants $M,\delta>0$ and $\tilde{\delta}\in\R$, see \cite[Section IV.16]{LSU68}.

Also, for $\al\in (0,\frac{1}{p})$, we define
\begin{align*}
P(t):=(-A(t))^{\al}_{|D},\qquad t\in [0,\tau].
\end{align*}
From \cite{AT87}, we deduce that there exists a constant $\tilde{M}>0$ such that for $0\le s<\tau'< \tau,$ and $\varphi\in D$,
\begin{align*}
\int_s^{\tau'} |(-A(t))^{\al}U(t,s)\varphi|^pdt &\le \tilde{M}\int_s^{\tau'} (t-s)^{-\al p}dt \|\varphi\|^p_X\cr & \le  \tilde{M} \frac{{\tau'}^{1-\al p}}{1-\al p} \|\varphi\|^p_X.
\end{align*}
This implies that $P(\cdot)$ satisfies the condition ${\bf (H2)}$ for $\mu=p$. We also define an observation operator
\begin{align*}
C:D\to \C,\quad C \varphi:=\varphi(c),\quad \varphi\in D.
\end{align*}
 As in \cite[p.23]{S02}, by using H\"older's inequality and the estimate \eqref{kernal},  for $\varphi\in D,$ $\tau' >s$, there exists a constant $c>0$ such that
\begin{align*}
\int_s^{\tau'} |CU(t,s)\varphi|^pdt & \le c \left(\int_s^{\tau'} (t-s)^{-\frac{np}{2q}}dt\right)\|\varphi\|^p_{L^q(\Om)} \cr & \le \ga^p \|\varphi\|^p_{L^q(\Om)}
\end{align*}
for $q>\frac{np}{2}$ and a certain constant $\ga>0$. This shows that $(C,A(\cdot))$ is admissible with exponent $p$.  Finally, according to Theorem \ref{main1}, $(C,A(\cdot)+P(\cdot))$ is admissible.
\subsection{A heat equation with mixed Dirichlet and Neumann conditions and a non local perturbation}
Let  $\Om$ be a bounded domain in $\R^n$ with $n\ge 2$ with smooth boundary $\partial\Om=\Gamma_0\cup \Gamma_1$ and $\Gamma_0\cap \Gamma_1=\emptyset$. We select
\begin{align*}
& X:=L^2(\Om),\quad Y:=L^2(\partial \Om), \cr & D:=\left\{ u\in H^2(\Om): u=0\;\text{on}\;\Gamma_0\;\text{and}\; \frac{\partial u}{\partial\nu}=0\;\text{on}\;\Gamma_1\right\},
\end{align*}
where $\nu$ is the unit normal vector. Define $\A:=\Delta$ with domain $D(\A)=D$. It is well known that $\A$ is a generator of a strongly continuous semigroup on $X$, see e.g. \cite{LTZ-04}. Now consider operators $A(t):D\to X,\,t\in [0,\tau]$ such that
\begin{align*}
A(t)u=\A u+\int_{\Gamma_1}\psi(t,\cdot,z)u(\cdot,z)dz,\qquad u\in D,
\end{align*}
where $\psi:[0,\tau]\times \Om\times \partial \Om\to \R$ is a measurable function such that $\psi(t,\cdot,\cdot)\in L^2(\Om\times \partial\Om))$ for a.e. $t\in [0,\tau]$ and \begin{align}\label{omega-estim-exa}
\|\psi(t,\cdot,\cdot)-\psi(s,\cdot,\cdot)\|_{L^2(\Om\times \partial\Om))}\le \om(t-s)
\end{align}
for a.e. $t,s\in [0,\tau]$, where $\om:[0,\tau]\to [0,\infty)$ is a continuous function with $w(0)=0$ and satisfies the condition \eqref{omega-condition} for $p=2$.

Let $b\in L^{\infty}([0,\tau])$ a scalar function and $\al\in (0,\frac{1}{2})$. Select
 \begin{align*}
P(t):=b(t)\A^{\al}:D\to X,\qquad t\in [0,\tau].
\end{align*}
Moreover, we consider the observation operator
\begin{align*}
C: D\to Y,\quad Cu=\begin{cases} 0,& \text{on}\,\Gamma_0,\cr u,& \text{on}\,\Gamma_1.\end{cases}
\end{align*}
According to \cite{LTZ-04}, $C\in\mathscr{O}^2_Y(\A)$. On the other hand, for any $t\in [0,\tau],$ we can write $A(t)=\A+B(t)C$ with $B(t):Y\to X$ defined by
\begin{align*}
B(t)g=\int_{\partial\Om}\psi(t,\cdot,z)g(\cdot,z)dz,\qquad g\in Y.
\end{align*}
Using \eqref{omega-estim-exa}, we obtain \begin{align*} \|B(t)g\|\le \left(\|\om\|_\infty +\|\psi(0,\cdot,\cdot)\|_{L^2(\Om\times \partial\Om))}\right)\|g\|_Y:= c \|g\|_Y\end{align*} for any $g\in Y$ and $t\in [0,\tau]$. Using the admissibility of $C$ for $\A$, for any $s,\tau'\in [0,\tau)$ such that $s<\tau'$, and $f\in D,$ we have
\begin{align*}
\int^{\tau'}_s \|B(t)Ce^{-\A (t-s)}f\|^2dt\le c^2 \int_0^{\tau'-s}\|Ce^{-\A t}f\|^2dt\le (c\ga)^2 \|f\|^2,
\end{align*}
for a constant $\ga>0$. Now by using \cite{H06}, we associate to $A(\cdot)$ an evolution family $U=(U(t,s))_{0\le s\le t\le \tau}$ on $X$. On the other hand, as $A\in \mathscr{MR},$ then by \cite{ABDH20}, $A(t)\in\mathscr{MR}$ for any $t\in [0,\tau]$. In addition, for each $t$, $B(t)C$ is admissible for the generator $\A,$ then by \cite{HI06}, $(C,A(t))$ is admissible for each $t\in [0,\tau].$ Thus by Theorem \ref{Kharou-thm2}, $(C,A(\cdot))$ is admissible. \\ Let us now discuss the admissibility of $C$ for $A(\cdot)+P(\cdot)$. For any $f\in D$ and $s,\tau'\in [0,\tau)$ with $s<\tau',$ and using H\"older's inequality, we have
\begin{align*}
\int^{\tau'}_s \|P(t)e^{-(t-s)\A}f\|^2&\le \|b(\cdot)\|_\infty\int^{\tau'-s}_0 \|(-\A)^{\al}e^{- t \A}f\|^2dt\cr & \le  \|b(\cdot)\|_\infty \int^{\tau'}_0 \frac{dt}{t^\al} \|f\|^2=  \|b(\cdot)\|_\infty \frac{(\tau')^{1-\al}}{1-\al} \|f\|^2\cr & \le 2\|b(\cdot)\|_\infty \tau' \|f\|^2
\end{align*}
for $\al\in (0,\frac{1}{2})$. Now by using Remark \ref{R1}, $P(\cdot)$ satisfies the condition {\bf(H2)} for $\mu=2$. On the other hand, for $t,s\in [0,\tau]$ and $f\in D,$ we have
\begin{align*}
\|A(t)f-A(s)f\|^2& \le\int_{\Om}\left(\int_{\partial \Om} |\psi(t,x,z)-\psi(s,x,z)| |(Cf)(z)|dz\right)^2dx\cr & \le
\int_{\Om} \int_{\partial \Om} |\psi(t,x,z)-\psi(s,x,z)|^2 dzdx \|Cf\|^2_Y\cr & \le \om(t-s)^2  \|Cf\|^2_Y,
\end{align*}
due to the estimate \eqref{omega-estim-exa}. This implies that $A(\cdot):[0,\tau]\to \calL(D,X)$ satisfies the relative $2$-Dini condition, so that the condition {\bf(H1)} holds. Thus $(C,A(\cdot)+P(\cdot))$ is admissible, due to Theorem \ref{main1}.

\section*{Data availability statement}
Data sharing not applicable to this article as no datasets were generated or analysed during the current study.

\end{document}